\documentclass[a4paper,11pt]{amsart}
\usepackage{hyperref}
\usepackage{todonotes}
\hypersetup{colorlinks=true,linkcolor=blue,citecolor=red}
\usepackage{comment}
\usepackage{graphicx} 
\usepackage{xcolor}
\usepackage{amssymb}
\usepackage{graphicx} 
\usepackage{tikz-cd}
\usepackage{stmaryrd}
\usepackage{amsmath}

\newcommand{\R}{\mathbb{R}}
\newcommand{\C}{\mathbb{C}}

\newcommand{\Z}{\mathbb{Z}}

\newcommand{\im}{\text{im}}

\newcommand{\del}{\partial}
\newcommand{\dbar}{\overline{\partial}}

\newcommand{\imm}{\operatorname{im}}

\newtheorem{lemma}{Lemma}[section]

\newtheorem{thm}{Theorem}[section]
\newtheorem*{thm*}{Theorem}

\title{Remarks on geometrically formal manifolds}

\usepackage{anysize}
\marginsize{2cm}{2cm}{2cm}{2cm}

\title[Non-trivial ABC-Massey products on complex parallelisable solvmanifolds]{Non-trivial ABC-Massey products on complex parallelisable solvmanifolds}

\author{Nunzia Cesarino}
\address[Nunzia Cesarino]{
Dipartimento di Matematica ``G. Vitali''\\
Università degli studi di Modena e Reggio Emilia\\
Via Giuseppe Campi, 213/b, 41125 Modena, Italy}
\email{nunzia.cesarino@unimore.it}

\author{Tommaso Sferruzza}
\address[Tommaso Sferruzza]{Dipartimento di Matematica ``Giuseppe Peano'' \\
Università di Torino\\
Via Carlo Alberto 10 \\
10123 Torino, Italy}
\email{tommaso.sferruzza@unito.it}

\author{Adriano Tomassini}
\address[Adriano Tomassini]{
Dipartimento di Scienze Matematiche, Fisiche e Informatiche\\
Unità di Mate\-matica e Informa\-tica\\
Università degli studi di Parma\\
Parco Area delle Scienze 53/A, 43124 Parma, Italy}
\email{adriano.tomassini@unipr.it}
\date{}

\thanks{The authors are partially supported by GNSAGA of INdAM. The second author was holder of a fully funded INdAM postdoctoral position.
The third author is partially supported by the Project PRIN 2022 ``Real and Complex Manifolds: Geometry and Holomorphic Dynamics 2022AP8HZ9''.}

\begin{document}

\begin{abstract}
    Triple Aeppli--Bott--Chern--Massey products, shortly, triple ABC-Massey products  are higher-order operations on the Bott--Chern and Aeppli cohomologies of a compact complex manifold, and their non-vanishing is an invariant of its pluripotential homotopy type. We prove that every non-Abelian complex unimodular solvable Lie algebra admits a non-vanishing triple ABC--Massey product, and we deduce that so does every compact complex parallelisable solvmanifold.
\end{abstract}
\maketitle
\tableofcontents

\section{Introduction}
Compact K\"ahler manifolds are the most important class of manifolds for which cohomological and metric properties interact in a deep and significant fashion. Indeed, 
on a compact K\"ahler manifold $X$ of complex dimension $n$, the Hodge decomposition
\[
H^{k}_{dR}(X;\C)\cong\bigoplus_{p+q=k}H^{p,q}_{\dbar}(X),\qquad \overline{H^{p,q}_{\dbar}(X)}\cong H^{q,p}_{\dbar}(X),
\]
holds, the Fr\"olicher spectral sequence degenerates at the first page, and the hard Lefschetz theorem forces strong numerical constraints on the Betti numbers. The relevant complex property satisfied by any compact K\"ahler manifold is the $\del\dbar$-Lemma, which also has purely homotopy-theoretic consequences: by \cite{DGMS}, the de Rham differential graded algebra of a compact K\"ahler manifold is formal in the sense of Sullivan \cite{Sull73}, so that its real homotopy type is determined by the cohomology ring and all Massey products vanish; likewise the Dolbeault bigraded algebra is formal in the sense of Neisendorfer and Taylor \cite{NT}.\\
On the other hand, many classes of compact complex manifolds are not K\"ahler. Among them, a key role is played by compact quotients of nilpotent, and more generally solvable, Lie groups, namely \emph{nilmanifolds} and \emph{solvmanifolds}, endowed with an invariant complex structure. Therefore, it is natural to study cohomological invariants of $X$ which obstruct the existence of further structures.\\
In this direction, Bott-Chern and Aeppli cohomologies provide useful tools in the study of compact non-K\"ahler complex manifolds, as we recall.\\
Let  $X=(M,J)$ be a compact complex manifold. The \emph{Bott-Chern} and \emph{Aeppli cohomologies} of $X$, defined by
\[
H_{BC}^{\bullet,\bullet}(X)=\frac{\ker \del \cap \ker \dbar \cap \mathcal{A}^{\bullet,\bullet}(X) }{\im\del\dbar\cap \mathcal{A}^{\bullet,\bullet}(X)}, \qquad H_{A}^{\bullet,\bullet}(X)=\frac{\ker \del\dbar \cap \mathcal{A}^{\bullet,\bullet}(X) }{(\im\del+\im \dbar)\cap \mathcal{A}^{\bullet,\bullet}(X)},
\]
play a fundamental role as holomorphic invariants in the study of the geometry of $X$ . On the one hand, for instance, their additive structure gives a numerical necessary and sufficient condition for $X$ to satisfy the $\del\dbar$-lemma \cite{AT-0} and they provide a finer holomorphic invariant than Dolbeault cohomology.

On the other hand, the multiplicative structure of the Bott-Chern and Aeppli cohomologies and their related homotopic  invariants have been intensively studied in the recent years, see, e.g., \cite{AT-1,Tar,ST22,Stelzig,PSZ,ST1,MS24,PSZ25,Rub26,ces,ST25,CRS,MagGS,SG26}.

A central role in \emph{pluripotential homotopy theory} \cite{Stelzig}, is played by the notion of triple ABC--Massey products, which had been introduced in \cite{AT-1} and generalized in \cite{Tar,MS24} as natural adaptation of classical Massey products to the context of Bott-Chern and Aeppli cohomologies. A similar construction for Dolbeault homotopy theory had been introduced in \cite{Tor}.\\
We briefly recall their definition.\\
Let $\mathfrak {a}_{12}=[\alpha_{12}]\in H_{BC}^{p,q}(X)$, $\mathfrak {a}_{23}=[\alpha_{23}]\in H_{BC}^{r,s}(X)$, and $\mathfrak {a}_{34}=[\alpha_{34}]\in H_{BC}^{u,v}(X)$ such that 
$$
\mathfrak {a}_{12}\cup\mathfrak {a}_{23}=0,\qquad \mathfrak {a}_{23}\cup\mathfrak {a}_{34}=0,
$$
that is, 
$$
(-1)^{p+q}\alpha_{12}\wedge\alpha_{23}=\del\dbar f_{13}, \qquad (-1)^{r+s}\alpha_{23}\wedge\alpha_{34}=\del\dbar f_{24}
$$
for some forms $f_{13},\,f_{24}$. Then, the \emph{triple $ABC$-Massey product} $\langle\mathfrak {a}_{12},\mathfrak {a}_{23},\mathfrak {a}_{34}\rangle_{ABC}$ is defined as
\begin{eqnarray*}
\langle\mathfrak {a}_{12},\mathfrak {a}_{23},\mathfrak {a}_{34}\rangle_{ABC}&=&[(-1)^{p+q}\alpha_{12}\wedge f_{24}-(-1)^{r+s}f_{13}\wedge\alpha_{34}]\\[5pt]
&\in& \frac{H_A^{p+r+u-1,q+s+v-1}(X)}{\mathfrak {a}_{12}\cup H_A^{r+u-1,s+v-1}(X)+ H_A^{p+r-1,q+s-1}(X)\cup\mathfrak {a}_{34}}.
\end{eqnarray*}

Together with their generalizations \cite{MS24, Tar}, non-vanishing triple ABC-Massey products  cohomological invariants define pluripotential higher order operations \cite{CRS} that obstruct the existence of special Hermitian metrics.\\
Indeed, let us fix $g$ a Hermitian metric on $X$. Let us denote by
\begin{equation*}
\Delta_{BC}:= \del\dbar\dbar^*\del^* +\dbar^*\del^*\del\dbar + \dbar^*\del \del^*\dbar + \del^*\dbar\dbar^*\del +\dbar^*\dbar +\del^*\del,
\end{equation*}
the \emph{Bott-Chern Laplacian}, a fourth order self-adjoint elliptic operator, and by
\[
\mathcal{H}_{\Delta_{BC}}^{\bullet,\bullet}(X)=\ker\Delta_{BC}\cap \mathcal{A}^{\bullet,\bullet}(X)
\]
the space of \emph{Bott-Chern harmonic forms}. In particular, Hodge Theory \cite{Schw} assures that
\[
H_{BC}^{\bullet,\bullet}(X)\cong \mathcal{H}_{\Delta_{BC}}^{\bullet,\bullet}(X).
\]
Then, a Hermitian metric $g$ on $X$ is said to be  \emph{geometrically Bott-Chern formal} if $(\mathcal{H}_{\Delta_{BC}}^{\bullet,\bullet}(X),\wedge)$ is an algebra \cite{AT-1}. The existence of such a metric restricts heavily the pluripotential homotopy type of $X$, as it forces every triple ABC-Massey product to vanish.

A key difference of pluripotential homotopy theory from the classical setting is that the non-vanishing of ABC--Massey products is compatible with the $\partial\bar{\partial}$-Lemma \cite{PSZ, ST22}. That is, compact complex manifolds satisfying the $\partial\bar{\partial}$-Lemma may still carry non-trivial ABC--Massey products. Furthermore, there exist classes of manifolds on which every triple ABC-Massey product vanishes, although they are not classically formal. For instance, nilmanifolds are not formal in the sense of Sullivan \cite{Hasegawa 2,BG1} nor Dolbeault formal if they are endowed with an invariant complex structure \cite{SV26}, yet there exist families of nilmanifolds with invariant complex structures with geometrically Bott-Chern formal metrics, see \cite{ST1} and Section 3. These features strongly motivate the study of ABC--Massey products and pluripotential homotopy theory in complex geometry.\\
\indent In this paper, we investigate the formality properties of a natural class of compact non-K\"{a}hler complex manifold, namely, compact quotients $\Gamma\backslash G$ of simply connected $2n$-dimensional Lie groups $G$ by a lattice $\Gamma$, endowed with an invariant complex structure $J$. Such manifolds provide a particularly interesting setting in literature, since many cohomological computations can be reduced to algebraic linear equations \cite{Ang,AK}.\\
We initially assume that $G$ is nilpotent and $h_{BC}^{1,0}(X)=n-1$. By working at the level of the Lie algebra $\mathfrak{g}$ of $G$, we first prove (Lemma \ref{lemma:complexspecialtype}) that $J$ must be a \emph{special type complex structure}, i.e., $G$ admits a left-invariant coframe of $(1,0)$-forms $\{\phi^1,\dots, \phi^n\}$ such that
\begin{equation*}\begin{cases}
    d\phi^{j}=0,\qquad j\in \{1,\dots, n-1\}\\
    d\phi^{n}\in \text{Span}_{\mathbb{C}}\langle \phi^{ij}, \phi^{i\overline{j}} \rangle_{i,j=1,\dots, n-1}.
\end{cases}
\end{equation*}
By making use of a result from \cite{ST1}, we can relate the existence of geometrically Bott-Chern formal metrics and \emph{strong K\"ahler with torsion metrics} \cite{Bis} (shortly, \emph{SKT metrics}), i.e., Hermitian metrics $g$ whose fundamental forms $\omega(\cdot,\cdot)=g(J\cdot,\cdot)$ are $\del\dbar$-closed. SKT metrics are a natural generalization of the K\"ahler condition and hold a crucial role in the context of K\"ahler geometry with torsion. Here, we prove the following.
\begin{thm*}[Theorem \ref{thm:complex-special-type}]
Let $X=(\Gamma\backslash G,J)$ be a $2n$-dimensional nilmanifold endowed with an invariant complex structure $J$ and let $h^{1,0}_{BC}(X)=n-1$. Then, any invariant Hermitian metric on $(\Gamma\backslash G,J)$ is SKT if and only if it is geometrically Bott-Chern formal.
\end{thm*}
In the second part of the paper, we consider the class of \emph{compact complex parallelisable manifolds}. These are characterized as compact quotients of complex Lie groups $G$ by lattices $\Gamma$ \cite{Wang}. In particular, we will assume that the complex Lie groups $G$ are solvable, namely, we will focus on \emph{complex parallelisable solvmanifolds}.\\
The main result of this section proves that compact complex parallelisable solvmanifolds always carry non-trivial triple ABC--Massey products. We break the proof in two separate steps.  First, we focus on complex parallelisable nilmanifolds, i.e., complex parallelizable solvmanifolds $\Gamma\backslash G$ with $G$ a complex nilpotent Lie group. The argument for both cases relies on a general procedure to produce non-vanishing triple ABC-Massey products at the level of the Lie algebra of $G$. Then, thanks to \cite[Lemma 2.1]{ST1}, such a non vanishing triple ABC-Massey product corresponds to a non-vanishing one on the corresponding complex parallelizable solvmanifold. More precisely, we prove the following
\begin{thm*}[Theorem \ref{thm:ABC-unimod-abel-nil}]
Let $X$ be a compact complex parallelisable nilmanifold. Then $X$ admits a non-vanishing triple ABC--Massey product.
\end{thm*}
Then, we prove the existence of ABC-Massey products for complex unimodular solvable Lie algebras (Theorem \ref{thm:ABC-unimod}), which allows to prove the the main theorem
\begin{thm*}[Theorem \ref{thm:ABC-solv-unimod}]
Let $X$ be a compact complex parallelisable solvmanifold. Then, $X$ admits a non-vanishing triple ABC-Massey product.
\end{thm*}

We note that the above theorems generalize the computations of \cite{ST25}, where triple ABC--Massey products were constructed case-by-case on each element of the family, up to complex dimension 5.  Here instead, we give a general formula  to induce non vanishing triple ABC--Massey product on any complex parallelisable solvmanifold via left-invariant forms.

\indent The paper is organized as follows. In section \ref{sec:preliminaries}, we set the notations and recall the main facts of Hodge theory and 
cohomologies on solvable Lie algebras endowed with a complex structure. We also recall the definition of Aeppli-Bott-Chern-Massey products. 
Section \ref{sec:complexspecialtype} is dedicated to the study of cohomological conditions for the existence of geometrically Bott-Chern formal metrics on nilmanifolds endowed with a complex structure of special type. Section \ref{sec:complexparallelisable} is devoted to the construction of non-trivial ABC-Massey products on complex parallelisable solvmanifolds. Finally, section \ref{sec:examples} contains the explicit construction of a class of complex parallelizable solvmanifolds and their non-vanishing triple ABC-Massey products.\vspace{1.0cm}

\noindent {\em Acknowledgements.} The authors would like to thank the Rényi Institute in Budapest for its warm hospitality on the occasion of the conference ``Focused workshop on cohomological and metric properties of Hermitian and symplectic manifolds" (02/23/2026-02/27/2026), during which most of the present paper was produced.

\section{Preliminaries}\label{sec:preliminaries}
\indent Let $G$ be a real $2n$-dimensional, connected and simply-connected Lie group with Lie algebra $\mathfrak{g}$. A complex structure $J$ on $G$ is said to be \emph{left-invariant} if it is induced via left translation by an endomorphism $J\in\text{End}(\mathfrak{g})$ satisfying $J^{2}=-id$ and the integrability condition
\[
N_J(u,v)=[u,v]+J[Ju,v]+J[u,Jv]-[Ju,Jv]=0, \qquad u,v\in\mathfrak{g}.
\]
The complexification of the dual Lie algebra $\mathfrak{g}_{\mathbb{C}}^{\ast}=\mathfrak{g}^{\ast}\otimes\mathbb{C}$ decomposes as
$$\mathfrak{g}_{\mathbb{C}}^{\ast}=(\mathfrak{g}_{\mathbb{C}}^{\ast})^{1,0}\oplus(\mathfrak{g}_{\mathbb{C}}^{\ast})^{0,1},$$
where 
$(\mathfrak{g}_{\mathbb{C}}^{\ast})^{1,0}=\{\alpha-iJ\alpha: \alpha\in\mathfrak{g}^{\ast}\}$ and $(\mathfrak{g}_{\mathbb{C}}^{\ast})^{0,1}=\{\alpha+iJ\alpha: \alpha\in\mathfrak{g}^{\ast}\}$ are the $\pm i$-eigenspaces of the $\mathbb{C}$-linear extension of $J$ on $\mathfrak{g}^{\ast}_{\mathbb{C}}$.\\ Correspondingly, $J$ induces a natural bigrading on the exterior algebra $\bigwedge^{\bullet}\mathfrak{g}_{\mathbb{C}}^{\ast}$, namely $$\textstyle \bigwedge^{\bullet}\mathfrak{g}^{\ast}_{\mathbb{C}}=\displaystyle\bigoplus_{p+q=\bullet}\textstyle\bigwedge^{p,q}\mathfrak{g}^{\ast},$$ where $\bigwedge^{p,q}\mathfrak{g}^{\ast}=\bigwedge^{p}(\mathfrak{g}_{\mathbb{C}}^{\ast})^{1,0}\otimes \bigwedge^{q}(\mathfrak{g}_{\mathbb{C}}^{\ast})^{0,1} $ denotes the space of $(p,q)$-forms on $\mathfrak{g}$. \vspace{0.1 cm}\\By the integrability of $J$, the Chevalley-Eilenberg differential $d$ on $\mathfrak{g}^{\ast}_{\mathbb{C}}$, defined by
\begin{equation*}
    d\alpha(u,v)=-\alpha([u,v]),\qquad \alpha\in\mathfrak{g}^{\ast}_{\mathbb{C}},\, u,v\in\mathfrak{g}_{\mathbb{C}},
\end{equation*}
decomposes as $d=\del+\dbar.$ \vspace{0.1 cm}\\ Therefore, the Bott-Chern and Aeppli cohomologies of $\mathfrak{g}$ are defined, respectively, by 
\begin{equation*}
    H_{BC}^{\bullet,\bullet}(\mathfrak{g}):=\frac{\ker \del \cap \ker\dbar\cap\textstyle\bigwedge^{\bullet,\bullet}\mathfrak{g}^{\ast}}{\im\, \del\dbar\cap\textstyle\bigwedge^{\bullet,\bullet}\mathfrak{g}^{\ast}},\qquad H_{A}^{\bullet,\bullet}(\mathfrak{g}):=\dfrac{\ker \del\dbar\cap\textstyle\bigwedge^{\bullet,\bullet}\mathfrak{g}^{\ast}}{(\im\, \del+\im\, \dbar)\cap\textstyle\bigwedge^{\bullet,\bullet}\mathfrak{g}^{\ast}}.
\end{equation*}
\indent Let $ \langle\,\,,\,\rangle $ be a Hermitian inner product on a unimodular Lie algebra $\mathfrak{g}$ endowed with a complex structure $J$. We denote by $\omega$ the associated fundamental form and by $\Omega$ the standard volume form, respectively defined by
\begin{equation*}
\omega(u,v):=\langle Ju,v\rangle ,\quad \forall \,u,\, v\in \mathfrak{g},\qquad \Omega=\frac{\omega^{n}}{n!}.
\end{equation*}
The $\mathbb{C}$-antilinear Hodge operator $\ast: \bigwedge^{p,q}\mathfrak{g}^{\ast}\to\bigwedge^{n-p,n-q}\mathfrak{g}^{\ast}$ with respect to $\langle\,\,,\,\rangle$ is defined by
\begin{equation*}
    \alpha\wedge \ast\beta=\langle \alpha,\beta\rangle\,\Omega,
\end{equation*}
where $\langle \alpha,\beta\rangle$ denotes the Hermitian inner product induced on $\bigwedge^{p,q}\mathfrak{g}^{\ast}$. We denote by $\del^{\ast}$ and $\dbar^{\ast}$ the adjoint operators of $\del$ and $\dbar$, respectively, with respect to the Hermitian inner product $\langle\,\,,\,\rangle$.\\
Since $\mathfrak{g}$ is unimodular, for any given $\beta\in\bigwedge^{p,q+1}\mathfrak{g}^{\ast}$, and any $\alpha\in \bigwedge^{p,q}\mathfrak{g}^{\ast}$, we have
\begin{align*}
    0=\langle d(\alpha\wedge\ast\beta), \Omega\rangle &= \langle \dbar(\alpha\wedge\ast\beta),\Omega\rangle\\
  &=\langle \dbar\alpha\wedge\ast\beta+(-1)^{p+q}\alpha\wedge\dbar\ast\beta,\Omega\rangle\\
 &= \langle \langle \dbar\alpha,\beta\rangle \Omega,\Omega\rangle+\langle\langle \alpha,\ast\,\dbar\ast\beta\rangle\Omega,\Omega\rangle,
\end{align*}
hence
\begin{equation*}
 \langle \dbar\alpha,\beta\rangle=\langle \alpha,-\ast\dbar\ast\beta\rangle.
\end{equation*}
Consequently, the adjoint operator of $\dbar$ is given by
 $$ \dbar^{\ast}=-\ast\dbar\ast.$$  
 Similarly, $\del^{\ast}=-\ast\del\,\ast$ holds.\\
According to \cite{Schw}, the \emph{Bott-Chern Laplacian} and the \emph{Aeppli Laplacian} on $\bigwedge^{\bullet,\bullet}\mathfrak{g}^{\ast}$ are the self-adjoint operators defined by
\begin{eqnarray*}
\Delta_{BC}&:=& \del\dbar\dbar^*\del^* +\dbar^*\del^*\del\dbar + \dbar^*\del \del^*\dbar + \del^*\dbar\dbar^*\del +\dbar^*\dbar +\del^*\del,\\
\Delta_{A}&:=& \del\del^*+\dbar\dbar^*+ \dbar^*\del^*\del\dbar +\del\dbar\dbar^*\del^*+\del\dbar^*\dbar\del^*+\dbar\del^*\del\dbar^*,
\end{eqnarray*}
whose kernels define, respectively, the spaces of \emph{Bott--Chern harmonic forms} and \emph{Aeppli harmonic forms}:
\begin{equation*}
 \mathcal{H}_{\Delta_{BC}}^{\bullet,\bullet}(\mathfrak{g})\   :=\ker\Delta_{BC}\vert_{\bigwedge^{\bullet,\bullet }\mathfrak{g}^{\ast}}\qquad \mathcal{H}_{\Delta_{A}}^{\bullet,\bullet}(\mathfrak{g}):=\ker\Delta_{A}\vert_{\bigwedge^{\bullet,\bullet}\mathfrak{g}^{\ast}}.\\\vspace{0.1 cm}
\end{equation*}
Since $\Delta_{BC}$ and $\Delta_{A}$ are self-adjoint operators, we have the following lemma. 
\begin{lemma}\label{HodgeDec}
Let $\mathfrak{g}$ be a unimodular Lie algebra endowed with a complex structure $J$ and a Hermitian inner product $\langle\,\,,\,\rangle$. Then, the following Hodge decompositions hold:
\begin{align*}
\textstyle\bigwedge^{p,q}\mathfrak{g}^*&= \mathcal{H}_{\Delta_{BC}}^{p,q}(\mathfrak{g})\   \oplus\imm\,\Delta_{BC}\vert_{\bigwedge^{p,q }\mathfrak{g}^{\ast}}\\
  &=\mathcal{H}_{\Delta_{BC}}^{p,q}(\mathfrak{g})\oplus^{\perp} \imm\,(\del\dbar_{|\bigwedge^{p-1,q-1}\mathfrak{g}^*})\oplus^{\perp} (\imm\,\del^{\ast}_{|\bigwedge^{p+1,q}\mathfrak{g}^*}+\imm\,\dbar^{\ast}_{|\bigwedge^{p,q+1}\mathfrak{g}^*}),
\end{align*}
and
\begin{align*}
\textstyle\bigwedge^{p,q}\mathfrak{g}^*&=\mathcal{H}_{\Delta_{A}}^{p,q}(\mathfrak{g})\oplus\imm\,\Delta_{A}\vert_{\bigwedge^{p,q }\mathfrak{g}^{\ast}}\\
  &=\mathcal{H}_{\Delta_{A}}^{p,q}(\mathfrak{g})\oplus^{\perp} (\imm\,\del_{|\bigwedge^{p-1,q}\mathfrak{g}^*}+\imm\,\dbar_{|\bigwedge^{p,q-1}\mathfrak{g}^*})\oplus^{\perp}\imm((\del\dbar)^{\ast}_{|\bigwedge^{p+1,q+1}\mathfrak{g}^*}).
\end{align*}
\end{lemma}
\noindent Accordingly, the following isomorphisms of complex vector spaces hold:
$$
\mathcal{H}_{\Delta_{BC}}^{\bullet,\bullet}(\mathfrak{g})\cong H^{\bullet,\bullet}_{BC}(\mathfrak{g}),\qquad \mathcal{H}_{\Delta_{A}}^{\bullet,\bullet}(\mathfrak{g})\cong H^{\bullet,\bullet}_{A}(\mathfrak{g}).
$$
In particular, $\alpha\in\bigwedge^{p,q}\mathfrak{g}^*$ is {\em Bott-Chern harmonic}  (respectively, {\em Aeppli harmonic}) if and only if
\begin{equation*}
\alpha\in \mathcal{H}_{\Delta_{BC}}^{p,q}(\mathfrak{g})\iff\begin{cases}
    \del\alpha=0\\
    \dbar\alpha=0\\
     \del\dbar\ast\alpha=0
\end{cases},\qquad 
\alpha\in \mathcal{H}_{\Delta_{A}}^{p,q}(\mathfrak{g})\iff\begin{cases}
   \del\dbar\alpha=0\\
   \del\ast\alpha=0\\
   \dbar\ast\alpha=0. 
\end{cases}
\end{equation*}
\indent Finally, we recall the definition of \emph{triple ABC-Massey products} on a Lie algebra $\mathfrak{g}$.\\
Let $\mathfrak {a}_{12}=[\alpha_{12}]\in H_{BC}^{p,q}(\mathfrak{g})$, $\mathfrak {a}_{23}=[\alpha_{23}]\in H_{BC}^{r,s}(\mathfrak{g})$, and $\mathfrak {a}_{34}=[\alpha_{34}]\in H_{BC}^{u,v}(\mathfrak{g})$ such that 
$$
\mathfrak {a}_{12}\cup\mathfrak {a}_{23}=0,\qquad \mathfrak {a}_{23}\cup\mathfrak {a}_{34}=0,
$$
that is, 
$$
(-1)^{p+q}\alpha_{12}\wedge\alpha_{23}=\del\dbar f_{13}, \qquad (-1)^{r+s}\alpha_{23}\wedge\alpha_{34}=\del\dbar f_{24}
$$
for some forms $f_{13}\in\bigwedge^{p+r-1,q+s-1}\mathfrak{g}^*,\,f_{24}\in\bigwedge^{r+u-1,s+v-1}\mathfrak{g}^*$. \\Then, the \emph{triple $ABC$-Massey product} $\langle\mathfrak {a}_{12},\mathfrak {a}_{23},\mathfrak {a}_{34}\rangle_{ABC}$ is defined as
\begin{eqnarray*}
\langle\mathfrak {a}_{12},\mathfrak {a}_{23},\mathfrak {a}_{34}\rangle_{ABC}&=&[(-1)^{p+q}\alpha_{12}\wedge f_{24}-(-1)^{r+s}f_{13}\wedge\alpha_{34}]\\[5pt]
&\in& \frac{H_A^{p+r+u-1,q+s+v-1}(\mathfrak{g})}{\mathfrak {a}_{12}\cup H_A^{r+u-1,s+v-1}(\mathfrak{g})+ H_A^{p+r-1,q+s-1}(\mathfrak{g})\cup\mathfrak {a}_{34}}.
\end{eqnarray*}
\indent Now, let $\Gamma$ be a lattice in $G$. Then $M=\Gamma\backslash G$ is a compact smooth manifold, and $J$ descends to a complex structure on $M$. Thus, we can consider the complex compact manifold $$X=(\Gamma\backslash G, J).$$
If, in addition, $G$ is solvable or nilpotent, the corresponding compact complex manifold $X=(\Gamma\backslash G, J)$ is called a \emph{solvmanifold} or \emph{nilmanifold}, respectively.
\section{Special type complex structures and geometric Bott-Chern formality}\label{sec:complexspecialtype}
Let $G$ be a real $2n$-dimensional, connected and simply-connected Lie group with Lie algebra $\mathfrak{g}$. Let $\Gamma$ be a lattice of $G$. An invariant complex structure $J$ on $X=(\Gamma\backslash G, J)$ is called a \emph{special type complex structure} \cite{ST1}, if there exists a coframe $\{\phi^{1},\dots, \phi^{n}\}\subset (g^{\ast}_{\mathbb{C}})^{1,0}$ satisfying
\begin{equation*}\begin{cases}
    d\phi^{j}=0,\qquad j\in \{1,\dots, n-1\}\\
    d\phi^{n}\in \text{Span}_{\mathbb{C}}\langle \phi^{ij}, \phi^{i\overline{j}} \rangle_{i,j=1,\dots, n-1}.
\end{cases}
\end{equation*}
The following lemma provides a cohomological condition in order that a complex structure on a nilpotent Lie algebra be of special type.
\begin{lemma}\label{lemma:complexspecialtype}
Let $\mathfrak{g}$ be a $2n$-dimensional nilpotent Lie algebra endowed with a complex structure $J$ such that $h_{BC}^{1,0}(\mathfrak{g})=n-1$. Then  $J$ is a special type complex structure on $\mathfrak{g}$.
\begin{proof}
Let $\{\phi^1,\dots, \phi^n\}$ be a basis of $(1,0)$-forms on $(\mathfrak{g},J)$, and let ${Z_1,\dots,Z_n}$ be its dual basis. By \cite{Sal01} and the assumption $h_{BC}^{1,0}(\mathfrak{g})=n-1$, we can assume that such a basis satisfies
\[
\begin{cases}
    d\phi^{i}=0, \quad i\in\{1,\dots,  n-1\}, \vspace{0.2cm}\\
    \displaystyle d\phi^n=\sum_{1\leq i<j\leq n} A_{ij}\phi^{ij}+\sum_{1\leq i\leq n-1}\sum_{1\leq j\leq n}B_{i\overline{j}}\phi^{i\overline{j}}.
\end{cases}
\]
Note that the commutator relations are then
\[
\begin{cases}[Z_i,Z_j]=-A_{ij}Z_n, \quad 1\leq i,j\leq n,\\
[\overline{Z_i},\overline{Z_j}]=-\overline{A_{ij}}\overline{Z_n}, \quad 1\leq i,j\leq n, \\
[Z_i,\overline{Z_{j}}]=-B_{i\overline{j}}Z_n+\overline{B_{j\overline{i}}}\,\overline{Z_n}, \quad 1\leq i\leq n-1, \quad  1\leq j\leq n-1 \\
[Z_i,\overline{Z_n}]=-B_{i\overline{n}}\,Z_n, \quad 1\leq i\leq n-1\\
[Z_n,\overline{Z_i}]=\overline{B_{i\overline{n}}}\,\overline{Z_n}, \quad 1\leq i\leq n-1.
\end{cases}
\]
In particular, we have $[Z_n,\overline{Z_n}]=0$, since $B_{n\overline{n}}=0$.\\
We claim that $A_{in}=0$ for every $i\in\{1,\dots, n-1\}$ and $B_{i\overline{n}}=0$ for every $i\in\{1,\dots, n-1\}$. \\
\indent By contradiction, assume that there exists $A_{in}\neq 0$, for $i\in\{1,\dots, n-1\}$, which implies that $[Z_i,Z_n]=-A_{in}Z_n$, i.e., $\mathfrak{g}_{\C}^{(1)}\supset\langle Z_n\rangle$. A simple induction argument shows then that $\mathfrak{g}_\C^{(k)}\supset \langle Z_n\rangle$ for every $k$, i.e., $\mathfrak{g}$ is not nilpotent, which is a contradiction. As a consequence, $A_{in}=0$ for every $i\in\{1,\dots, n-1\}$.\\
\indent Let us now evaluate $d^2\phi^n$ on the vectors $Z_i,Z_n,\overline{Z_i}$, $i\in\{1,\dots, n-1\}$. From the Jacobi identity,
\begin{align*}
0&=d^2\phi^n(Z_i,Z_n,\overline{Z_i})=d(d\phi^n)(Z_i,Z_n,\overline{Z_i})\\
&=-d\phi^n([Z_i,Z_n],\overline{Z_i})+d\phi^n([Z_i,\overline{Z_i}],Z_n)-d\phi^n([Z_n,\overline{Z_i}],Z_i)\\
&=-d\phi^n(B_{i\overline{i}}Z_n-\overline{B_{i\overline{i}}}\,\overline{Z_n},Z_n)-\overline{B_{i\overline{n}}}d\phi^n(\overline{Z_n},Z_i)\\
&=\overline{B_{i\overline{i}}}d\phi^n(\overline{Z_n},Z_n)-\overline{B_{i\overline{n}}}\phi^n([Z_i,\overline{Z_n}])\\
&=-\overline{B_{i\overline{n}}}\phi^n([Z_i,\overline{Z_n}])=|B_{i\overline{n}}|^2.
\end{align*}
This implies that $B_{i\overline{n}}=0$ for every $i\in\{1,\dots, n-1\}$.

To summarize, the structure equations become
\[
\begin{cases}
    d\phi^{i}=0, \quad i\in\{1,\dots,  n-1\}, \vspace{0.2cm}\\
    \displaystyle d\phi^n=\sum_{1\leq i<j< n} A_{ij}\phi^{ij}+\sum_{1\leq i, j\leq n-1}B_{i\overline{j}}\phi^{i\overline{j}},
\end{cases}
\]
hence the complex structure is of special type.
\end{proof}
\end{lemma}
Combining the previous lemma and \cite[Remark 3.3]{ST1}, we obtain the following theorems, which relate geometrically Bott-Chern formal metrics to SKT ones. We first recall their definitions on the Lie algebra $\mathfrak{g}$.\\
A Hermitian metric $g$ is said to be \emph{geometrically Bott-Chern formal} if $(\mathcal{H}_{\Delta_{BC}}^{\bullet,\bullet}(\mathfrak{g}),\wedge)$ is an algebra, while it is said to be \emph{strong K\"{a}hler with torsion}, or shortly \emph{SKT}, if the fundamental form $\omega$ satisfies $\del\dbar\omega=0.$  
\begin{thm}\label{thm:5.1}
Let $\mathfrak{g}$ be a $2n$-dimensional nilpotent Lie algebra endowed with a complex structure $J$ such that $h_{BC}^{1,0}(\mathfrak{g})=n-1$. Then, any Hermitian metric on $(\mathfrak{g},J)$ is SKT if and only if it is geometrically Bott-Chern formal.
\end{thm}
Since any special type complex structure $J$ is nilpotent, by Theorem 3.8 \cite{Ang}, the following isomorphism $$H_{BC}^{\bullet,\bullet}(\mathfrak{g})\simeq H_{BC}^{\bullet,\bullet}(X)$$
holds. In particular, for any invariant Hermitian metric on $X$, it holds that $$\mathcal{H}^{\bullet,\bullet}_{\Delta_{BC}}(\mathfrak{g})=\mathcal{H}^{\bullet,\bullet}_{\Delta_{BC}}(X).$$
As a consequence, by Theorem \ref{thm:5.1}, it follows that
\begin{thm}\label{thm:complex-special-type}
Let $X=(\Gamma\backslash G,J)$ be a $2n$-dimensional nilmanifold endowed with an invariant complex structure $J$ and let $h^{1,0}_{BC}(X)=n-1$. Then, any invariant Hermitian metric on $(\Gamma\backslash G,J)$ is SKT if and only if it is geometrically Bott-Chern formal.
\end{thm}


\section{Complex parallelisable manifolds}\label{sec:complexparallelisable}

\indent A compact complex manifold $X$ is \emph{complex parallelisable} \cite{Wang} if it has a trivial holomorphic tangent bundle. Equivalently, $X$ is complex parallelisable if it is biholomorphic to a complex coset
space $$X\simeq D\backslash G$$ of a complex Lie group $G$ over a discrete subgroup $D$ \cite[Theorem 1]{Wang}. Here henceforth, we will consider $X=D\backslash G$ to be a complex parallelizable solvmanifold, i.e., $G$ is a complex solvable Lie group and $\Gamma:=D\subset G$ is a lattice. Moreover, we will always assume that $G$ is not Abelian, so that $X$ is not a complex torus.\\
Let $X$ be a complex parallelisable solvmanifold of complex dimension $n$ and let us  denote by $r=\text{dim} H_{BC}^{1,0}(\mathfrak{g})$. Since $G$ is solvable not Abelian, we have $1\le r<n$. By Lie's theorem, there exists a basis $\{\phi^{1},\dots, \phi^{n}\}\subset \mathfrak{g}^{\ast}$ such that
\begin{equation}\label{eq:cplx-par-solv}
\begin{cases}
    d\phi^\nu=0, \quad \nu\in\{1,\dots, r\}\\
    d\phi^\nu=\xi_{\nu}\wedge \phi^\nu+\eta_{\nu}, \quad \nu\in\{r+1,\dots, n\},
\end{cases}
\end{equation}
where \[\xi_{\nu}=\sum_{\mu=1}^ra_{\nu\mu}\phi^{\mu},\] and $\eta_{\nu}$ is a linear combination of wedge products of the preceding basis elements $\phi^{1},\dots,\phi^{\nu-1}$; see \cite[§2]{N}. In particular, $$d\xi_\nu=0,\qquad \nu\in\{r+1,\dots, n\}.$$
Starting from \eqref{eq:cplx-par-solv}, in the next two subsections we will prove that  there always exists a non-trivial triple ABC-Massey product on complex parallelisable solvmanifolds. In particular, by \cite[Lemma 2.1.]{ST1}, there exists an injective inclusion map:
\[
\iota:
\left\{
\text{triple \(ABC\)-Massey products on } (\mathfrak{g},J)
\right\}
\hookrightarrow
\left\{
\text{triple \(ABC\)-Massey products on } (\Gamma\backslash G,J)
\right\},
\]
which allows one to prove the existence of non-trivial triple ABC-Massey products on the corresponding compact complex manifold by means of a computation at the level of left-invariant forms.

We will first consider the special case of complex parallelisable nilmanifolds. The case of complex parallelisable solvmanifolds will be treated in the subsequent section.
\subsection{Complex parallelisable nilmanifolds} 
\begin{thm}\label{thm:ABC-unimod-abel-nil}
Let $X$ be a complex parallelisable nilmanifold. Then $X$ admits a non-vanishing triple ABC-Massey product.
\end{thm}

\begin{proof}
Let $\{\phi^1,\dots, \phi^n\}$ be a basis of $(1,0)$-forms satisfying equations \eqref{eq:cplx-par-solv}. In the nilpotent case, we have 
$$\quad \xi_\nu=0, \quad \forall \nu\in\{r+1,\dots, n\}.$$
Then, the structure equations become
\begin{equation*}
\begin{cases}
    d\phi^\nu=0, \quad \nu\in\{1,\dots, r\}\\
    d\phi^\nu=\eta_{\nu}, \quad \nu\in\{r+1,\dots, n\}.
\end{cases}
\end{equation*}
Since $\{\phi^1,\dots, \phi^n\}$ is a complex $(1,0)$-coframe and $\phi^j$ is holomorphic for $j=1,\ldots,n$, we can consider the following Bott-Chern cohomology classes
$$
\mathfrak{a}_{12}=[\alpha_{12}]=[\del\phi^{r+1\cdots n}]\in H_{BC}^{n-r+1, 0}(\mathfrak{g}),\quad\quad \mathfrak{a}_{23}=[\alpha_{23}]=[\dbar\phi^{\overline{r+1}\cdots \overline{n}}]\in H_{BC}^{0, n-r+1}(\mathfrak{g}).
$$
Note that we can decompose $d\phi^{r+1\cdots n}$ as the sum of two components
$$ d\phi^{r+1\cdots n}=(d\phi^{r+1})\wedge\phi^{r+2\cdots n}-\phi^{r+1}\wedge d\phi^{r+2\cdots n}.$$
Since the first term on the right hand side does not contain the form $\phi^{r+1}$, the two terms are linearly independent.
Since $d\phi^{r+1}\wedge\phi^{r+2\cdots n}\ne 0$, the form $\alpha_{12}$ is non-trivial. \\
Furthermore, as a consequence of the structure equations, we can also take
$$
\mathfrak{a}_{34}=[\alpha_{34}]=[\phi^{\overline{1}\cdots \overline{r}}]\in H_{BC}^{0, r}(\mathfrak{g}).
$$
Since 
$$
\alpha_{12}\wedge \alpha_{23}=\del\dbar\,(-1)^{n-r}\phi^{r+1\cdots n\,\overline{r+1}\cdots \overline{n}}\,,
$$
and, due to a bidegree issue, 
$$ \alpha_{23}\wedge \alpha_{34}=0,$$
we get
$$\mathfrak{a}_{12}\cup \mathfrak{a}_{23}=0,\quad \mathfrak{a}_{23}\cup\mathfrak{a}_{34}=0.$$
Therefore, the following $ABC$-Massey product is well-defined:
\begin{eqnarray*}
\langle\mathfrak {a}_{12},\mathfrak {a}_{23},\mathfrak {a}_{34}\rangle_{ABC}&=&[(-1)^{n-r+1}\,\phi^{r+1\cdots n\,\overline{r+1}\cdots \overline{n}} \wedge \phi^{\overline{1}\cdots \overline{r}}]=[(-1)^{n-r+1}(-1)^{(n-r)r}\,\phi^{r+1\cdots n\,\overline{1}\cdots \overline{n}}]\\[5pt]
&\in& \frac{H_A^{n-r,n}(\mathfrak{g})}{H_A^{n-r, n-r}(\mathfrak{g})\cup[\phi^{\overline{1}\dots \overline{r}}]}.
\end{eqnarray*}
We will show that 
$$
[(-1)^{(n-r)(r+1)+1}\,\phi^{r+1\cdots n\,\overline{1}\cdots \overline{n}}]\not\in H_A^{n-r, n-r}(\mathfrak{g})\cup[\phi^{\overline{1}\dots \overline{r}}],
$$
so that the triple $ABC$-Massey product $\langle\mathfrak {a}_{12},\mathfrak {a}_{23},\mathfrak {a}_{34}\rangle_{ABC}$ is non-vanishing on $\mathfrak{g}$.\\\\
By contradiction, suppose that there exists an Aeppli harmonic form
$$\xi\in \textstyle\bigwedge^{n-r, n-r}\mathfrak{g}^{\ast},$$
and forms $R\in\textstyle\bigwedge^{n-r-1,n}\mathfrak{g}^{\ast},\quad S\in \textstyle\bigwedge^{n-r,n-1}\mathfrak{g}^{\ast}$ such that\\
\begin{equation}\label{nilpeq1}
(-1)^{(n-r)(r+1)+1}\,\phi^{r+1\cdots n\,\overline{1}\cdots \overline{n}}=\xi\wedge \phi^{\overline{1}\dots \overline{r}}+\del R+\dbar S.
\end{equation}
\\
Let $g=\sum_{j=1}^n\phi^j\otimes \overline{\phi^j}$ be the Hermitian metric on $\mathfrak{g}$, we denote by $\omega$ its associated fundamental form and by $\Omega=\frac{\omega^{n}}{n!}$ the standard volume form.\\
Multiplying equation \eqref{nilpeq1} by $$\ast(\phi^{r+1\cdots n\,\overline{1}\cdots \overline{n}})=\epsilon\,\phi^{1\cdots r},$$ where $\epsilon\in\mathbb{R}\setminus\{0\};$
we get\\ 
\begin{equation}\label{nilpeq2}
    C\,\Omega= \big( \xi\wedge \phi^{\overline{1}\dots \overline{r}}+\del R+\dbar S\big)\wedge\ast(\phi^{r+1\cdots n\,\overline{1}\cdots \overline{n}}),\qquad C\in\mathbb{C}\setminus\{0\}.
\end{equation}
Since $d\phi^{1}=\dots =d\phi^{r}=0$, it is immediate to see that
$$d^{\ast}(\phi^{r+1\cdots n\,\overline{1}\cdots \overline{n}})=0,$$  which implies that $\phi^{r+1\cdots n\,\overline{1}\cdots \overline{n}}$ is Aeppli-harmonic. As a result, we can rewrite the last equation as
\begin{equation*}
    C\,\Omega= \xi\wedge \phi^{\overline{1}\cdots \overline{r}}\wedge\ast(\phi^{r+1\cdots n\,\overline{1}\cdots \overline{n}})+d\, P,
\end{equation*}
where 
$$P=\big( R+ S\big)\wedge\ast(\phi^{r+1\cdots n\,\overline{1}\cdots \overline{n}}).$$
However, since $\mathfrak{g}$ is nilpotent and therefore unimodular, $d\vert_{\bigwedge^{2n-1}\mathfrak{g}^{\ast}}\equiv 0$. Thus, $dP=0$ and
\begin{equation}\label{nilpeq3}
    C\,\Omega= \xi\wedge \phi^{\overline{1}\cdots \overline{r}}\wedge\ast(\phi^{r+1\cdots n\,\overline{1}\cdots \overline{n}}).
\end{equation}
Furthermore, we can write the right-hand side as:
$$
 \xi\wedge \phi^{\overline{1}\cdots \overline{r}}\wedge\ast(\phi^{r+1\cdots n\,\overline{1}\cdots \overline{n}})=\epsilon\, \xi\wedge \phi^{\overline{1}\cdots\overline{r}}\wedge\phi^{1\cdots r}=(-1)^{r^{2}}\epsilon\,\phi^{1\cdots r\,\overline{1}\cdots \overline{r}}\wedge\ast(\ast\xi).
$$
We can take the inner product of both sides of \eqref{nilpeq3} by $\Omega$
\[
0\neq C\langle\Omega,\Omega\rangle=\langle (-1)^{r^{2}}\epsilon\,\phi^{1\cdots r\,\overline{1}\cdots \overline{r}}\wedge\ast(\ast\xi),\Omega\rangle=(-1)^{r^{2}}\epsilon\,\langle \phi^{1\cdots r\overline{1}\cdots \overline{r}},\ast \xi\rangle \langle\Omega,\Omega\rangle.
\]
Since $C$ and $\epsilon$ are nonzero constants and $\Omega$ is a volume form, it suffices to prove the following claim in order to obtain a contradiction.
\\\textbf{Claim:} $$\langle \phi^{1\cdots r\,\overline{1}\cdots \overline{r}},\, \ast\xi\rangle=0.$$
From Lemma \ref{HodgeDec}, we have that
$$\mathcal{H}_{\Delta_{BC}}^{n-r,n-r}(\mathfrak{g})\perp\imm(\del\dbar_{|\bigwedge^{r-1,r-1}\mathfrak{g}^*}).$$
Therefore, we want to show 
\begin{equation*}
    \phi^{1\cdots r\,\overline{1}\cdots\overline{r}}\in \text{im}(\del\dbar_{|\bigwedge^{r-1,r-1}\mathfrak{g}^*}),
\end{equation*}
as $\ast\xi\in\mathcal{H}_{\Delta_{BC}}^{r,r}(\mathfrak{g}).$\\
From the structure equations, we have
\begin{equation*}
    d\phi^{r+1}=\sum_{1\le i<j\leq r}A_{ij}\,\phi^{i}\wedge\phi^{j},
\end{equation*}
and, in particular, we can suppose $A_{lm}\ne 0$ for some $1\le l<m\le r$. Then, by a direct computation, we get
\begin{align*}
    \del\dbar\phi^{1\cdots \hat{l}\cdots \hat{m}\cdots r\, r+1\,\overline{1} \cdots  \hat{\overline{l}}\cdots \hat{\overline{m}}\cdots \overline{r}\,\overline{r+1}}&=\\
    &=(-1)^{r-1} \del\phi^{1\cdots \hat{l}\cdots \hat{m}\cdots r\, r+1}\wedge\dbar\phi^{\overline{1} \cdots\hat{\overline{l}}\cdots\hat{\overline{m}} \cdots \overline{r}\,\overline{r+1}}\\
    &=(-1)^{r-1}|A_{lm}|^{2}\phi^{1\cdots r\,\overline{1}\cdots \overline{r}},
\end{align*}
where the symbol $\hat{\,}$ denotes that the corresponding index is omitted. Therefore, in this way, we can write
$$ \phi^{1\cdots r\, \overline{1}\cdots \overline{r}}=  \del\dbar\, \big(\dfrac{(-1)^{r-1}}{|A_{lm}|^{2}}\, \phi^{1\cdots \hat{l}\cdots \hat{m}\cdots r\, r+1\,\overline{1} \cdots  \hat{\overline{l}}\cdots \hat{\overline{m}}\cdots \overline{r}\,\overline{r+1}}\big).$$
\end{proof}
\subsection{Complex solvable Lie algebras}
We now consider the more general case of complex parallelisable solvmanifolds. First, we work at the level of the Lie algebra (complex unimodular solvable Lie algebras) and then deduce the corresponding result for compact solvmanifolds.
\begin{thm}\label{thm:ABC-unimod}
Let $\mathfrak{g}$ be any complex unimodular solvable
Lie algebra. Then $\mathfrak{g}$ admits a non-vanishing triple ABC-Massey product.
\end{thm}
Firstly, we give an explicit proof of \cite[Lemma 1.3]{N} which will be useful in the proof of the Theorem.
\begin{lemma}\label{Lemma2}
\[
\sum_{\nu=r+1}^n\xi_{\nu}=0.
\]
\begin{proof}
Let $\{\phi^1,\dots, \phi^n\}$ be a basis of $(1,0)$-forms satisfying equations \eqref{eq:cplx-par-solv}, and let ${Z_1,\dots,Z_n}$ be its dual basis. We note that each $\xi_{\nu}$ can be written as
\[
\xi_{\nu}=\sum_{\mu=1}^r B_{\mu}^{\nu}\phi^{\mu}, \qquad B_{\mu}^{\nu}=\phi^{\nu}([Z_\nu,Z_\mu])=-\phi^{\nu}(\text{ad}_{Z_{\mu}}(Z_{\nu})).
\]
Then
\begin{align*}
\sum_{\nu=r+1}^n\xi_{\nu}&=\sum_{\nu=r+1}^n\left(\sum_{\mu=1}^r B_{\mu}^\nu \phi^\mu\right)=\sum_{\mu=1}^r\left(\sum_{\nu=r+1}^nB_{\mu}^\nu\right)\phi^\mu\\
&=-\sum_{\mu=1}^r\left(\sum_{\nu=r+1}^n\phi^{\nu}(\text{ad}_{Z_{\mu}}(Z_{\nu}))\right)\phi^\mu\\
&=-\sum_{\mu=1}^r\text{tr}(\text{ad}_{Z_\mu})\phi^\mu\\
&=0,
\end{align*}
since $\mathfrak{g}$ is unimodular.
\end{proof}
\end{lemma}

\begin{proof}
Let ${\phi^1,\dots,\phi^n}$ be a basis of $(1,0)$-forms satisfying equations \eqref{eq:cplx-par-solv}. Let us assume that $\nu_0\in\{r+1,\dots, n\}$ is the lowest index such that
\[
0\neq \xi_{\nu_0}=\sum_{\mu=1}^ra_{\nu_0\mu}\phi^{\mu}, \qquad a_{\nu_0\mu}\in \C.
\]
Otherwise, the Lie algebra $\mathfrak{g}$ is nilpotent, and the proof follows from Theorem \ref{thm:ABC-unimod-abel-nil}.\\
Since $\{\phi^1,\dots, \phi^{n}\}$ is a complex $(1,0)$-coframe and $\phi^j$ is holomorphic for $j=1,\ldots,n$, we can consider the following Bott-Chern cohomology classes
$$
\mathfrak{a}_{12}=[\alpha_{12}]=[\del\phi^{\nu_0+1\cdots n}]\in H_{BC}^{n-\nu_0+1, 0}(\mathfrak{g}),\quad\quad \mathfrak{a}_{23}=[\alpha_{23}]=[\dbar\phi^{\overline{\nu_0+1}\cdots \overline{n} }]\in H_{BC}^{0, n-\nu_0+1}(\mathfrak{g}).
$$
and
$$
\mathfrak{a}_{34}=[\alpha_{34}]=[\phi^{\overline{1}\cdots \overline{\nu_0}}]\in H_{BC}^{0, \nu_0}(\mathfrak{g}).
$$
Note that, by Lemma \ref{Lemma2},  $\alpha_{12}$ can be written as
\begin{align*}
\alpha_{12}=\del \phi^{\nu_0+1\cdots n}&=\left(\sum_{\nu=\nu_0+1}^n\xi_\nu\right) \wedge \phi^{\nu_0+1\cdots n}+\sum_{\nu=\nu_0+1}^n(-1)^{\nu_0+1+\nu}\eta_{\nu}\wedge \phi^{\nu_0+1\cdots \hat{\nu}\cdots n}\\
& =-\xi_{\nu_0}\wedge \phi^{\nu_0+1\cdots n}+\sum_{\nu=\nu_0+1}^n(-1)^{\nu_0+1+\nu}\eta_{\nu}\wedge \phi^{\nu_0+1\cdots \hat{\nu}\cdots n},
\end{align*}
since $\sum_{\nu=r+1}^n \xi_{\nu}=0$ and $\xi_\nu=0$ for $\nu\in\{r+1,\cdots, \nu_{0}-1\}$.

The first term of the last expression is linearly independent from all the summands on the right, as none of terms of the form $\eta_\nu\wedge \phi^{\nu_0+1\cdots \hat{\nu}\cdots n}$ contains the term $\phi^{\nu_0+1\cdots n}$. Moreover, $\xi_{\nu_0}\wedge\phi^{\nu_0+1\cdots n}\neq 0$, so that
\[
\alpha_{12}\neq 0,
\]
and, consequently, also $\alpha_{23}=\overline{\alpha_{12}}\neq 0$. Moreover, from structure equations, it follows that $d\alpha_{34}=0$.
\\ 
By a direct calculation, it follows
$$
\alpha_{12}\wedge \alpha_{23}=\del\dbar\,(-1)^{n-\nu_0}\phi^{\nu_0+1\cdots n\,\overline{\nu_0+1}\cdots \overline{n}}\,,
$$
and $\alpha_{23}\wedge \alpha_{34}=0$, by bidegree.
As a consequence, 
$$\mathfrak{a}_{12}\cup \mathfrak{a}_{23}=0,\quad \mathfrak{a}_{23}\cup\mathfrak{a}_{34}=0,$$hence, the following $ABC$-Massey product is well-defined:
\begin{eqnarray*}
\langle\mathfrak {a}_{12},\mathfrak {a}_{23},\mathfrak {a}_{34}\rangle_{ABC}&=&[(-1)^{(n-\nu_0)(\nu_0+1)+1}\,\phi^{\nu_0+1\cdots n} \wedge \phi^{\overline{1}\cdots \overline{n}}]\\[5pt]
&\in& \frac{H_A^{n-\nu_0,n}(\mathfrak{g})}{H_A^{n-\nu_0, n-\nu_0}(\mathfrak{g})\cup[\alpha_{34}]}.
\end{eqnarray*}
In order to prove that the triple $ABC$-Massey product $\langle\mathfrak {a}_{12},\mathfrak {a}_{23},\mathfrak {a}_{34}\rangle_{ABC}$ is non-vanishing on $\mathfrak{g}$, we will show that 
$$
[(-1)^{(n-\nu_0)(\nu_0+1)+1}\,\phi^{\nu_0+1\cdots n} \wedge  \phi^{\overline{1}\cdots \overline{n}}]\not\in H_A^{n-\nu_0, n-\nu_0}(\mathfrak{g})\cup[\alpha_{34}].
$$
Let $g=\sum_{j=1}^{n}\phi^j\otimes \overline{\phi^j}$ be the Hermitian metric on $\mathfrak{g}$, we denote by $\omega$ its associated fundamental form and by $\Omega=\frac{\omega^{n}}{n!}$ the standard volume form.\\
By contradiction, suppose that there exists an Aeppli harmonic form
$$\theta\in \textstyle\bigwedge^{n-\nu_0,\, n-\nu_0}\mathfrak{g}^{\ast},$$
and forms $R\in \textstyle\bigwedge^{n-\nu_0-1,n}\mathfrak{g}^{\ast},\quad S\in \textstyle\bigwedge^{n-\nu_0,n-1}\mathfrak{g}^{\ast}$ such that\\
\begin{equation}\label{eqfij}
(-1)^{(n-\nu_0)(\nu_0+1)+1}\,\phi^{\nu_0+1\cdots n\, \overline{1}\cdots \overline{n}}=\theta\wedge \alpha_{34}+\del R+\dbar S.
\end{equation}
\\
We multiply equation \eqref{eqfij} by
$$\ast\phi^{\nu_0+1\cdots n\,\overline{1}\dots \overline{n}}=\epsilon\,\phi^{1\cdots\nu_0}=\epsilon\,\overline{\alpha_{34}}, \qquad \epsilon\in\mathbb{R}\setminus\{0\},$$
to obtain
\begin{equation}\label{eq2}
    C\,\Omega=\big( \theta\wedge \alpha_{34} +\del R+\dbar S\big)\wedge \ast(\phi^{\nu_0+1\cdots n\, \overline{1}\cdots \overline{n}}),
\end{equation}
where $C\in\mathbb{C}\setminus\{0\}$.\\\vspace{0.1 cm}
Since $d\alpha_{34}=d\,\overline{\alpha_{34}}=0$, it is immediate to see that 
$$d^{\ast}(\phi^{\nu_0+1\cdots n\,\overline{1}\dots \overline{n}})=0,$$ 
which implies $\phi^{\nu_0+1\cdots n\,\overline{1}\dots \overline{n}}$ is Aeppli-harmonic.
As a result, equation \eqref{eq2} becomes
\begin{equation*}
  C\,\Omega=  \theta\wedge \alpha_{34}\wedge \ast(\phi^{\nu_0+1\cdots n\,\overline{1}\cdots \overline{n}})+dP
\end{equation*}
where 
$$P=\big( R+ S\big)\wedge \ast(\phi^{\nu_0+1\cdots n\,\overline{1}\cdots \overline{n}} ).$$
Since $\mathfrak{g}$ is unimodular by hypothesis, we have $d\vert_{\bigwedge^{2n-1}\mathfrak{g}^{\ast}}\equiv 0$, and hence $dP=0$. Therefore, 
\begin{align}\label{eq3}
   C\,\Omega&=  \theta\wedge  \alpha_{34}\wedge \ast(\phi^{\nu_0+1\cdots n\,\overline{1}\cdots \overline{n}})\\
   &=\epsilon\,\theta\wedge  \alpha_{34}\wedge \overline{ \alpha_{34}}\nonumber\\
   &=\epsilon(-1)^{\nu_{0}^{2}}\,\theta\wedge \phi^{1\cdots \nu_0\,\overline{1}\cdots \overline{\nu_0}}\nonumber
\end{align}
We take the inner product of equation \eqref{eq3} by $\Omega$, and we obtain
\[
0\neq C||\Omega||^2=\epsilon(-1)^{\nu_{0}^{2}}\,\langle \phi^{1\cdots \nu_0\,\overline{1}\cdots \overline{\nu_0}}\wedge \ast (\ast \theta),\Omega\rangle=\epsilon(-1)^{\nu_{0}^{2}}\,\langle \phi^{1\cdots \nu_0\,\overline{1}\cdots \overline{\nu_0}},\ast\theta\rangle||\Omega||^2.
\]
\noindent {\bf Claim.}
\[
\langle\phi^{1\cdots \nu_0\,\overline{1}\cdots \overline{\nu_0}},\ast \theta\rangle=0.
\]
In order to prove the claim, we will show that 
\begin{equation*}
 \phi^{1\cdots \nu_0\,\overline{1}\cdots \overline{\nu_0}}  \in \del\dbar(\textstyle\bigwedge^{\nu_0-1,\nu_0-1}\mathfrak{g}^{\ast}),
\end{equation*}
which will lead to a contradiction, as $C$ and $\epsilon$ are non-zero constants, $\Omega$ is a volume form, and by Lemma \ref{HodgeDec} we have that
$$\mathcal{H}_{\Delta_{BC}}^{\nu_0,\nu_0}(\mathfrak{g})\perp \text{im}(\del\dbar_{|\bigwedge^{\nu_0-1,\nu_0-1}\mathfrak{g}^{\ast}}).$$
Indeed, since
\begin{equation*}
\xi_{\nu_0}\ne 0,
\end{equation*}
there exists $\mu_0\in\{1,\dots, r\}$ such that
\[
a_{\nu_0\mu_0}\neq 0.
\]
By a direct computation, we get
\begin{equation*}
 \del\dbar\, \phi^{1\cdots\hat{\mu_0}\cdots  \nu_0\, \overline{1} \cdots\hat{\overline{\mu_0}}\cdots \overline{\nu_0}}=(-1)^{\nu_0-1}\del\,\phi^{1\cdots\hat{\mu_0}\cdots  \nu_0}\wedge \dbar\,\phi^{\overline{1} \cdots\hat{\overline{\mu_0}}\cdots \overline{\nu_0}}=(-1)^{\nu_0-1}|a_{\nu_0\mu_0}|^2\phi^{1\cdots \nu_0\, \overline{1}\cdots \overline{\nu_0}}.
\end{equation*}
Hence,
\begin{equation*}
    \phi^{1\cdots \nu_0\, \overline{1}\cdots \overline{\nu_0}}=\del\dbar\, \big(\frac{(-1)^{\nu_0-1}}{|a_{\nu_0\mu_0}|^{2}}  \,  \phi^{1\cdots\hat{\mu_0}\cdots  \nu_0\, \overline{1} \cdots\hat{\overline{\mu_0}}\cdots \overline{\nu_0}}\big).
\end{equation*}
\end{proof}
As a consequence, we obtain the following.
\begin{thm}\label{thm:ABC-solv-unimod}
Let $X$ be a compact complex parallelisable solvmanifold. Then, $X$ admits a non-vanishing triple ABC-Massey product.
\end{thm}


\section{Examples}\label{sec:examples}
In this section, we provide an explicit family of compact complex parallelisable solvmanifolds to which Theorem \ref{thm:ABC-solv-unimod} applies. More precisely, we construct a family of compact quotients of semidirect products of the form $\C\ltimes_{\rho}\C^{2n}$ by exhibiting suitable lattices.\\
\indent Let $M\in \text{SL}(2;\Z)$ be a matrix with positive eigenvalues $e^{\lambda}, e^{-\lambda}\in \R$. In particular, there exists $P\in\text{GL}(2;\R)$ such that
\[
P\cdot M\cdot P^{-1}=\text{diag}(e^{\lambda}, e^{-\lambda})=:D.
\]
Let us fix the complex Lie group $$G=:\C\ltimes_{\rho} \C^{2n},$$ where
\[
\rho(z)=\text{diag}(e^{\lambda z}, e^{-\lambda z}\dots, e^{\lambda z}, e^{-\lambda z}), \qquad z\in \C,
\]
so that the product $\ast$ on $G$ is given by
\begin{align*}
(y_0,y_1,&\dots, y_{2n})\ast (z_0,z_1,\dots, z_{2n})\\
&=(z_0+y_0,e^{\lambda y_0}z_1+y_1,e^{-\lambda y_0}z_2+y_2 \dots, e^{\lambda y_{0}}z_{2n-1}+y_{2n-1},e^{-\lambda y_0}z_{2n}+y_{2n}).  
\end{align*}
Now, we want to construct a splitting type lattice $$\Gamma=\Gamma'\ltimes_{\rho} \Gamma''.$$ First, we choose  $\Gamma'\subset\mathbb{C}$ so that the image under $\rho$ is given by integer powers of $D$. Thus, for $a_0\in\mathbb Z$, we require
\[
e^{\lambda z}=e^{\lambda a_0},
\]
which gives
\[
\lambda z-\lambda a_0=2\pi \sqrt{-1} b_0,
\qquad b_0\in\mathbb Z,
\]
and hence
\[
z=a_0+\frac{2\pi \sqrt{-1}}{\lambda}b_0.
\]
Therefore, we set
\[
\Gamma':=\left\{\left(a_0+\sqrt{-1}\frac{2\pi}{\lambda} b_0\right): a_0,b_0\in \Z \right\}\subset \C.
\]
Then, for every $z\in \Gamma'$,
\[
e^{\lambda z}=e^{\lambda a_{0}}e^{2\pi\sqrt{-1}b_{0}}=e^{\lambda a_{0}},
\] 
and similarly, $e^{-\lambda z}=e^{-\lambda a_{0}}.$
Thus, 
\[
\rho(z)=\text{diag}(e^{\lambda a_{0}}, e^{-\lambda a_{0}}, \dots, e^{\lambda a_{0}}, e^{-\lambda a_{0}})=\text{diag}(D^{a_{0}},\dots, D^{a_{0}}).
\]
We now define the lattice in the second factor $\mathbb{C}^{2n}$ by
\[\Gamma'':=\left\{\left( ^t(P\begin{pmatrix}
    a_1 + \sqrt{-1}b_1\\
    a_2+\sqrt{-1}b_2
\end{pmatrix}),\dots,  ^t(P\begin{pmatrix}
    a_{2n-1} + \sqrt{-1}b_{2n-1}\\
    a_{2n}+\sqrt{-1}b_{2n}
\end{pmatrix})\right) \in\C^{2n}: a_{i},b_{i}\in\mathbb{Z},\, i=1,\dots,2n \right\}\subset \C^{2n}.
\]
The condition $\rho(\Gamma')(\Gamma'')\subset\Gamma''$ follows from the relation $DP=PM,$ which yields
\[
D^{a_{0}}P= PM^{a_{0}}.
\]
Since $M\in \text{SL}(2;\Z)$, we have 
\[
M^{a_{0}}(\mathbb{Z}^{2})=\mathbb{Z}^{2}.
\]
Therefore, 
\[
D^{a_{0}}(P(\mathbb{Z}^{2}))=PM^{a_{0}}(\mathbb{Z}^{2})=P(\mathbb{Z}^{2}).
\]
Hence, $\Gamma=\Gamma'\ltimes_{\rho} \Gamma''$ is a lattice in $G$.\\
\indent Denoting by $J$ the invariant complex structure induced by $G$, the manifold $X=(\Gamma\backslash G,J)$ is a compact complex parallelisable manifold of complex dimension $2n+1$.\\
The Lie algebra $\mathfrak{g}$ of $G$ splits as $\C\ltimes_{\gamma} \C^{2n}$, where $\gamma$ is infinitesimally induced by the map $\rho$. From the expression of $\rho$ it is easy to see that the forms
\[
\phi^0:=dz_0, \qquad \phi^{2i-1}:=e^{-\lambda z_0}dz_{2i-1}, \qquad \phi^{2i}:=e^{\lambda z_0}dz_{2i}, \qquad i\in\{1,\dots, n\},
\]
are a basis of left-invariant $(1,0)$-forms on $G$, which defines a basis of $(\mathfrak{g}_{\C}^*)^{1,0}$ and, moreover, it descends to $X$. In particular, the structure equations are
\[
d\phi^0=0, \qquad d\phi^{2i-1}=-\lambda\phi^{0}\wedge\phi^{2i-1}, \qquad d\phi^{2i}=\lambda\phi^{0}\wedge\phi^{2i}, \qquad i\in\{1,\dots, n\}.
\]
Finally, as an application of Theorem \ref{thm:ABC-solv-unimod}, we obtain
\begin{align*}
\langle [\lambda\,\phi^{02\cdots 2n}],[\lambda\,\phi^{\overline{0}\,\overline{2}\cdots\overline{2n}}],&[\phi^{\overline{0}\,\overline{1}}]\rangle_{ABC}\\
&=[\phi^{2\cdots 2n\,\overline{0}\,\overline{1}\,\overline{2}\cdots\overline{2n}}]_{A}\in \frac{H_A^{2n-1,2n+1}(X)}{H_A^{2n-1, 2n-1}(X)\cup[\phi^{\overline{0}\,\overline{1}}]},
\end{align*}
which is an explicit non-vanishing triple ABC-Massey product on the compact complex parallelisable manifold $X=(\Gamma\backslash G,J)$.

\end{document}